\documentclass{amsart}
\usepackage[a4paper,marginratio=1:1]{geometry}
\usepackage{amssymb}
\usepackage[initials,non-sorted-cites]{amsrefs}
\usepackage[symbol,perpage]{footmisc}

\numberwithin{equation}{section}
\usepackage{mathtools}\mathtoolsset{showonlyrefs}

\usepackage{mathrsfs,amsmath,braket}

\newtheorem{thm}{Theorem}[section]
\newtheorem{prop}[thm]{Proposition}

\newtheorem{cor}[thm]{Corollary}
\theoremstyle{definition}
\newtheorem{dfn}[thm]{Definition}
\theoremstyle{remark}
\newtheorem{rem}[thm]{Remark}
\newtheorem{ex}[thm]{Example}

\DeclareMathOperator{\Imaginary}{Im}

\DeclareMathOperator{\aut}{Aut}
\DeclareMathOperator{\Id}{id}
\DeclareMathOperator{\vol}{vol}
\DeclareMathOperator{\sgn}{sgn}
\DeclareMathOperator{\spn}{Span}

\def\R{\mathbb{R}}
\def\C{\mathbb{C}}

\def\Z{\mathbb{Z}}

\def\H{\mathbb{H}}
\def\N{\mathbb{N}}

\def\dt{\frac{\partial}{\partial t}}

\def\map{\longrightarrow}

\newcommand{\kak}[1]{\left(#1\right)}

\newcommand{\conj}[1]{\overline{#1}}
\newcommand{\FS}[1]{\mathscr{L}_#1}
\newcommand{\abs}[1]{\lvert #1 \rvert}

\title{The Folland--Stein spectrum of some Heisenberg Bieberbach manifolds}
\author{Yoshiaki Suzuki}
\address{Department of Mathematical Science, Faculty of Fundamental Sciences, Graduate School of Science and Technology, Niigata University, Niigata 950-2181, Japan}
\email{f21j010a@mail.cc.niigata-u.ac.jp}

\begin{document}

\maketitle

\begin{abstract}
We study the eigenvalues and eigenfunctions of the Folland--Stein operator $\mathscr{L}_\alpha$ on some examples of 3-dimensional Heisenberg Bieberbach manifolds, that is, compact quotients $\Gamma\backslash\H$ of the Heisenberg group $\H$ by  a discrete torsion-free subgroup $\Gamma$ of $\H\rtimes U(1)$.
\end{abstract}

\section{Introduction}
The aim of this paper is to study the eigenvalues and eigenfunctions of the Folland--Stein operator $\mathscr{L}_\alpha$ on some examples of what we call Heisenberg Bieberbach manifolds. Let $\H$ be the 3-dimensional Heisenberg group with the standard left-invariant CR structure.
As is well known, the Kohn Laplacian acting on differential forms on $\H$ is, on the level of coefficient functions, given by the Folland--Stein operators $\mathscr{L}_\alpha=\mathscr{L}_0+i\alpha T$, where $\mathscr{L}_0$ is the sub-Laplacian, $T$ is the Reeb vector field and $\alpha\in\R$.
A \emph{Heisenberg Bieberbach manifold} is a compact quotient $\Gamma\backslash\H$ by a discrete torsion-free subgroup $\Gamma$ of $\H\rtimes U(1)$, called a \emph{Heisenberg Bieberbach group}, where $U(1)$ is identified with unitary automorphisms of $\H$.
Since unitary automorphisms preserve the CR structure on $\H$ and the Folland--Stein operator is invariant under unitary automorphisms, any Heisenberg Bieberbach manifold inherits a natural CR structure and the Folland--Stein operator $\mathscr{L}_\alpha$ descends to it.

In the case where $\Gamma=N$ is a lattice subgroup of $\H$, Folland \cite{folland} studied the operator $\mathscr{L}_\alpha$ on $L^2(N\backslash\H)$.
His idea was using the Weil--Brezin transform to explicitly decompose the right translation representation of $\H$ on $L^2(N\backslash\H)$ into irreducible representations.
He showed that the image of an eigenfunction of the Hamiltonian of the harmonic oscillator on $\R$ by the Weil--Brezin transform is an eigenfunction of the sub-Laplacian, thereby completely determining the eigenvalues and eigenfunctions of $\mathscr{L}_\alpha$ on $L^2(N\backslash\H)$.
We are interested to see how well Folland's method can be extended to Heisenberg Bieberbach manifolds.

In this paper, we treat the two examples $\Gamma_{2l,\pi}$ and $\Gamma'_{2l,\frac{\pi}{2}}$ of Heisenberg Bieberbach groups which are not lattices. The group $\Gamma_{2l,\pi}$ contains a certain lattice $N_{2l}$ as its normal subgroup of $\Gamma_{2l,\pi}$ of index 2 and is generated by $N_{2l}$ and a certain element $\varphi \in \H \rtimes U(1)$.
We want to find $\varphi$-invariant eigenfunctions on $N_{2l}\backslash\H$ because such functions can be regarded as eigenfunctions on $\Gamma_{2l,\pi}\backslash\H$. Folland's result gives us the explicit forms of the eigenfunctions on $N_{2l}\backslash\H$, in terms of which
we will describe the eigenspaces on $\Gamma_{2l,\pi}\backslash\H$ and calculate their dimensions.
In this way, roughly speaking, it will be shown that the dimension of each eigenspace on $\Gamma_{2l,\pi}\backslash\H$ is about half of the dimension of the corresponding eigenspaces on $N_{2l}\backslash\H$ (see Theorem \ref{cor:dim}).
In particular, an analogue of Weyl's law holds for $\Gamma_{2l,\pi}\backslash\H$.
The same argument can be applied to the other example $\Gamma'_{2l,\frac{\pi}{2}}$ to show an analogues result in Section \ref{sectionGammaprime}.
The computation is substantially more complicated than in the previous example.

We remark that Heisenberg Bieberbach groups are analogues of Bieberbach groups in the Euclidean case, that is,  discrete torsion-free subgroups $\Pi\subset\R^3\rtimes O(3)$ with compact quotient $\Pi\backslash\R^3$.
The generalized notion of \emph{almost-Bieberbach groups} including these Euclidean and Heisenberg cases has been studied by some authors including Auslander \cite{auslander} and Dekimpe \cite{dekimpe}, and corresponding quotient manifolds, \emph{infra-nilmanifolds}, have attracted interest mainly from the viewpoints of algebraic and differential topology (see for instance \cite{dekimpe} and \cite{dekimpe2}).
However, analytic approaches toward infra-nilmanifolds have been scarce, which is why we find our problem particularly interesting.

\section{Heisenberg manifolds}

\subsection{The Heisenberg group and lattice subgroups}
We review some of the standard facts on the Heisenberg group and its lattice subgroups.
The \emph{Heisenberg group} is the Lie group $\H=\C\times\R$ with
the group multiplication given by
\begin{equation}
(z,t)\cdot(z',t')=(z+z',t+t'+2\Imaginary z\bar{z}')
\end{equation}
for $(z,t),(z',t')\in\H$.
One can easily see that the vector fields
\begin{align*}
X=\frac{\partial}{\partial x}+2y\dt,\quad
Y=\frac{\partial}{\partial y}-2x\dt,\quad
T=\dt
\end{align*}
form a basis of the space of left-invariant vector fields, where $z=x+iy$.
Let us consider the complex vector field on $\H$ given by
\begin{equation}
Z=\frac{1}{2}(X-iY)=\frac{\partial}{\partial z}+i\bar{z}T.
\end{equation}
The subbundle $T^{1,0}\H$ of $\C\otimes T\H$ spanned by $Z$ is a so-called strictly pseudoconvex CR structure on $\H$. Let $\wedge^{0,q}\H$ denote the space of sections of $\bigwedge^q(T^{0,1}\H)^*$ for $q=0,1$, where $T^{0,1}\H=\conj{T^{1,0}\H}$.
The tangential $\bar{\partial}$ operator
$\bar{\partial}_b\colon C^{\infty}(\H)\rightarrow \wedge^{0,1}\H$ is given by
\begin{equation}
\bar{\partial}_bf=(\bar{Z}f)\,d\bar{z}.
\end{equation}
Moreover, the Kohn Laplacian is defined by
\begin{equation}
\Box_b=\bar{\partial}_b\bar{\partial}_b^*+\bar{\partial}_b^*\bar{\partial}_b\colon\wedge^{0,q}\H\rightarrow\wedge^{0,q}\H,\quad q=0,1.
\end{equation}
This boils down to $\Box_bf=\mathscr{L}_{1}f$ and  $\Box_b(f\,d\bar{z})=(\mathscr{L}_{-1}f)\,d\bar{z}$, where
we write
\begin{equation}
\mathscr{L}_\alpha=-\frac{1}{2}(Z\bar{Z}+\bar{Z}Z)+i\alpha T
\end{equation}
for any real number $\alpha$. We call $\mathscr{L}_\alpha$ the \emph{Folland--Stein operator}. For more details on the Heisenberg group, we refer the reader to \cite{follandstein}.

\begin{dfn}
A subgroup $N\subset\H$ is a \emph{lattice subgroup} if it is a discrete subgroup for which the quotient manifold
$M=N\backslash\H$ is compact.
A \emph{Heisenberg manifold} is the quotient of the Heisenberg group by a lattice subgroup.
\end{dfn}
Since any left-invariant vector field descends to $M$,
a Heisenberg manifold $M$ has the natural CR structure and the Folland--Stein operator $\mathscr{L}_\alpha$ can be considered as a differential operator on $M$.

To explain the classification of lattices, it is convenient to use a different group multiplication on $\H$. We define
\begin{equation}
(p,q,s)\star(p',q',s')=(p+p',q+q',s+s'+pq')
\end{equation}
for $(p,q,s),(p',q',s')\in\H$.
Then the group $(\H,\star)$ is called the \emph{polarized Heisenberg group} and the coordinate system $(p,q,s)$ is called the \emph{polarized coordinates}. The polarized Heisenberg group $(\H,\star)$ is isomorphic to the Heisenberg group $(\H,\cdot)$ that we defined in the beginning. Indeed, the map
\begin{equation}
(\H,\cdot)\ni(x,y,t)\mapsto\kak{y,x,\frac{t}{4}+\frac{xy}{2}}\in(\H,\star)
\end{equation}
is an isomorphism. In the following, we basically use the polarized Heisenberg group. With respect to the polarized coordinates, it is easy to check that the Folland--Stein operator $\mathscr{L}_\alpha$ is written as
\begin{equation}
\mathscr{L}_\alpha=\frac{1}{4}(-(P^2+Q^2)+i\alpha S), \label{fs}
\end{equation}
where
\begin{equation}
P=\frac{\partial}{\partial p},\quad Q=\frac{\partial}{\partial q}+p\frac{\partial}{\partial s},\quad S=\frac{\partial}{\partial s}.
\end{equation}

We now introduce a specific class of lattices of $\H$.
For $l\in\Z_{>0}$, the discrete subgroup
\begin{equation}
N_l=\Z\times l\Z\times\Z
\end{equation}
is a lattice. Indeed, the rectangle $Q_l=[0,1)\times[0,l)\times[0,1)$ is a fundamental domain for $N_l$. Since $Q_l$ has compact closure, the quotient $N_l\backslash\H$ is also compact.
The importance of the lattice $N_l$ lies in the following fact, whose proof can be found in \cite{tolimieri}.
\begin{prop}\label{classlattice}
For any lattice $N\subset\H$, there exist a unique $l\in\Z_{>0}$ and an automorphism $\Phi\in\aut_0(\H)$ such that $\Phi(N)=N_l$.
\end{prop}
Here, $\aut_0(\H)$ denotes the identity component of the group $\aut(\H)$ of Lie group automorphisms of $\H$. For any $\Phi\in\aut_0(\H)$, it is known that there exist a symplectic map $\Phi_1$, an inner automorphism $\Phi_2$ and a dilation $\Phi_3$ such that $\Phi=\Phi_1\Phi_2\Phi_3$ (see \cite{follandharm} for a proof), where
a symplectic map is
\begin{equation}
(p,q,s)\mapsto\kak{Cq+Dp, Aq+Bp, s+\frac{1}{2}((Aq+Bp)(Cq+Dp)-pq)}
\end{equation}
associated with a symplectic matrix
$
\kak{
\begin{array}{cc}
A & B \\
C & D \\
\end{array}
}
$
, an inner automorphism is
\begin{equation}
(p,q,s)\mapsto(a,b,c)(p,q,s)(a,b,c)^{-1}=(p,q,s+aq-bp)
\end{equation}
for $(a,b,c)\in\H$, and a dilation is $(p,q,s)\mapsto(rp,rq,r^2s)$ for $r>0$.

\subsection{Harmonic analysis on Heisenberg manifolds}
In this section, we review the method of harmonic analysis on Heisenberg manifolds following \cite{folland}.

Let $M$ be a Heisenberg manifold $N\backslash\H$ where $N$ is a lattice of $\H$. Note that the polarized coordinates equip $\H$ with the left-invariant Lebesgue measure, which therefore induces a measure on $M$.
We may identify the space of square integrable functions $L^2(M)$ with the space of $N$-invariant functions $f$ on $\H$ with $\int_Q|f|^2\,dp\,dq\,ds<\infty$, where $Q$ is a fundamental domain for $N$.
The right action of $\H$ on $M$ induces the unitary representation $R$ of $\H$ on $L^2(M)$ defined by
\begin{equation}
(R(h)f)([g])=f([gh]), \quad g,h\in\H.
\end{equation}
We want to study the decomposition of $R$ in terms of irreducible representations.

First, we expand a function $f(p,q,s)\in L^2(M)$ into the Fourier series with respect to the action of the center $Z(\H)=\set{(0,0,s)\mid s\in\R}$ of $\H$. Note that $Z_N=Z(\H)\cap N$ is a discrete subgroup of $Z(\H)$, and so we have $Z_N=\set{(0,0,cs)\mid s\in\Z}$ for some $c>0$.
Since $f$ is $N$- invariant, it follows that $f(p,q,s+c)=f(p,q,s)$, and hence $f$ can be expanded into the Fourier series in $s$. We thus get
\begin{equation}
L^2(M)=\bigoplus_{n\in\Z}\mathscr{H}_n, \label{LM=Hn}
\end{equation}
where
$\mathscr{H}_n$ is the subspace of functions $f\in L^2(M)$ of the form $f(p,q,s)=e^{\frac{2\pi in}{c}s}g(p,q)$, where $g\in L^2(\R^2/\Lambda_N)$, $\Lambda_N$ being the image of the projection $N\ni(p,q,s)\mapsto(p,q)\in\R^2$.

A function in $\mathscr{H}_0$ can be regarded as a function on the torus $\R^{2}/\Lambda_N$. Hence the irreducible decomposition of $\mathscr{H}_0$ is given by
\begin{equation}
\mathscr{H}_0=\bigoplus_{(\mu,\nu)\in\Lambda'_N}\C\chi_{\mu,\nu}, \label{H0=Cxi}
\end{equation}
where $\chi_{\mu,\nu}(p,q)=e^{2\pi i(\mu p+\nu q)}$ and $\Lambda'_N=\set{(\mu,\nu)\in\R^2 |
\mu u+\nu v\in\Z \ \text{for all}\ (u,v)\in\Lambda_N }$.

For $n\neq 0$, we see at once that $R|_{\mathscr{H}_n}$ is unitarily equivalent to a sum of copies of the \emph{Schr\"{o}dinger representation} $\pi_{\frac{n}{c}}$ of $\H$ on $L^2(\R)$ by the Stone--von Neumann theorem (see \cite{follandharm}). Here, the Schr\"{o}dinger representation $\pi_{\beta}$ for $\beta\in\R\setminus\{0\}$ is defined by
\begin{equation}
(\pi_\beta(p,q,s)g)(x)=e^{2\pi i\beta(s+qx)}g(x+p)
\end{equation}
for $g\in L^2(\R)$.
Moreover, if $N=N_l$ for some $l\in\Z_{>0}$, the Weil--Brezin transform defined immediately below precisely describes the equivalence between $R|_{\mathscr{H}_n}$ and a sum of copies of $\pi_{n}$.
Note that $Z_{N_l}=\set{(0,0,s)\mid s\in\Z}$, and so $\mathscr{H}_n=\set{f\in L^2(M)\mid f(p,q,s)=e^{2\pi ins}g(p,q)}.$
For $a\in\Z/\abs{n}\Z$ and $b\in\Z/l\Z$, we define the \emph{Weil--Brezin transform} $W_n^{a,b}\colon L^2(\R)\map\mathscr{H}_n$ by
\begin{equation}
(W_n^{a,b}g)(p,q,s)=e^{2\pi ins}\sum_{k\in\Z}g\kak{p+k+\frac{j}{\abs{n}}+\frac{u}{ln}}e^{2\pi in\kak{k+\frac{j}{\abs{n}}+\frac{u}{ln}}q}, \label{WB}
\end{equation}
where $0\leq j\leq|n|-1$ with $a\equiv j\bmod\abs{n}$ and $0\leq u\leq l-1$ with $b\equiv u\bmod l$. We write $\mathscr{H}_n^{a,b}=\Imaginary W^{a,b}_n$. Using the Weil--Brezin transform, we can decompose $\mathscr{H}_n$ explicitly by the following proposition. A proof can be found in \cite{folland}.
\begin{prop}[Brezin \cite{brezin}]\label{weilbrezin}
The Weil--Brezin transform $W_n^{a,b}\colon L^{2}(\R)\map\mathscr{H}^{a,b}_n$ is a unitary transform which satisfies that  $R(h) \circ W_n^{a,b}=W_n^{a,b} \circ \pi_n(h)$ for $h\in\H$, and we have
\begin{equation}
\mathscr{H}_n=\bigoplus_{n\neq0,a\in\Z/\abs{n}\Z,b\in\Z/l\Z}\mathscr{H}_n^{a,b}. \label{H_n=H_n^ab}
\end{equation}
In particular, $R|_{\mathscr{H}_n^{a,b}}\simeq\pi_n$, and the multiplicity of $\pi_n$ in $R$ is $l\abs{n}$.
\end{prop}
By \eqref{LM=Hn}, \eqref{H0=Cxi} and \eqref{H_n=H_n^ab} we have the irreducible decomposition of $L^2(N_l\backslash\H)$ given by
\begin{equation}
L^2(N_l\backslash\H)=\bigoplus_{(\mu,\nu)\in\Lambda'_{N_l}}\C\chi_{\mu,\nu}\oplus\bigoplus_{n\neq0,a\in\Z/|n|\Z,b\in\Z/l\Z}\mathscr{H}_n^{a,b}.
\end{equation}

In general, for any lattice $N$, we can also obtain the irreducible decomposition of $L^2(N\backslash\H)$ using the argument for the case of the specific class of lattices. Suppose that $Z_N=\set{(0,0,cs)\mid s\in\Z}$. By Proposition \ref{classlattice},  there exist a unique $l\in\Z_{>0}$ and an automorphism $\Phi\in\aut_0(\H)$ such that $\Phi(N)=N_l$.
 First, we have $L^2(N_l\backslash\H)=\mathscr{H}_0\oplus\bigoplus_{n,a,b}\mathscr{H}_n^{a,b}$ and $\pi_n\simeq R|_{\mathscr{H}^{a,b}_n}$ by Proposition \ref{weilbrezin}. In particular, $\pi_n\circ\Phi\simeq (R\circ\Phi)|_{\mathscr{H}^{a,b}_n}$.
Using an expression $\Phi=\Phi_1\Phi_2\Phi_3$ mentioned above,  we can easily see that the Jacobian $|d\Phi|$ of $\Phi$ is constant. Hence the map $\Phi^*\colon F\mapsto |d\Phi|^{\frac{1}{2}}(F\circ\Phi)$ is a unitary map from $\mathscr{H}^{a,b}_n$ onto its image denoted by $\widetilde{\mathscr{H}}^{a,b}_n\subset L^2(N\backslash\H)$, and so $\Phi^*$ gives the equivalence $(R\circ\Phi)|_{\mathscr{H}^{a,b}_n}\simeq R|_{\widetilde{\mathscr{H}}^{a,b}_n}$.
On the other hand, since $Z_N=\set{(0,0,cs)\mid s\in\Z}$, the dilation component $\Phi_3$ is
$(p,q,s)\mapsto(c^{-\frac{1}{2}}p,c^{-\frac{1}{2}}q,c^{-1}s)$
, and hence
$(\pi_n\circ\Phi)(0,0,s)=e^{\frac{2\pi ins}{c}}\Id.$ By the Stone--von Neumann theorem, we see that
$\pi_{\frac{n}{c}}\simeq\pi_n\circ\Phi$. Therefore $\pi_{\frac{n}{c}}\simeq \pi_n\circ\Phi\simeq R|_{\widetilde{\mathscr{H}}^{a,b}_n}$.
Summarizing, we have the following.
\begin{cor}\label{coro:LNH}
 Let $N$ be a lattice of $\H$ with $Z_N=\set{(0,0,cs)\mid s\in\Z}$. Suppose that $\Phi(N)=N_l$ for some $l\in\Z_{>0}$ and $\Phi\in\aut_0(\H)$. Then we have $\pi_{\frac{n}{c}}\simeq R|_{\widetilde{\mathscr{H}}^{a,b}_n}$ and
\begin{equation}
L^2(N\backslash\H)=\bigoplus_{(\mu,\nu)\in\Lambda'_{N}}\C\chi_{\mu,\nu}\oplus\bigoplus_{n\neq0,a\in\Z/\abs{n}\Z,b\in\Z/l\Z}\widetilde{\mathscr{H}}_n^{a,b}, \label{LNH}
\end{equation}
where $\widetilde{\mathscr{H}}^{a,b}_n=\Phi^*(\mathscr{H}^{a,b}_n)$. In particular, the multiplicity of $\pi_{\frac{n}{c}}$ in $R$ is $l\abs{n}$.
\end{cor}

Now we can find the spectrum of the Folland--Stein operator $\mathscr{L}_\alpha$. Let $N$ be a lattice of $\H$ with $Z_N=\set{(0,0,cs)\mid s\in\Z}$. Recall that the Folland--Stein operator $\mathscr{L}_\alpha$ is written as
\eqref{fs}.
For the $n=0$ component in \eqref{LNH}, we have
\begin{equation}
\mathscr{L}_\alpha\chi_{\mu,\nu}=\pi^2(\mu^2+\nu^2)\chi_{\mu,\nu},\quad(\mu,\nu)\in\Lambda'_N.
\end{equation}
Hence $\chi_{\mu,\nu}$ is an eigenfunction of $\mathscr{L}_\alpha$ corresponding to the eigenvalue $\pi^2(\mu^2+\nu^2)$.

For the $n\neq 0$ components, it is easy to check that
\begin{equation}
d\pi_{\frac{n}{c}}(P)=\frac{d}{dx},\quad d\pi_{\frac{n}{c}}(Q)=\frac{2\pi in}{c}x,\quad d\pi_{\frac{n}{c}}(S)=\frac{2\pi in}{c}
\end{equation}
on $L^2(\R)$. Therefore $\mathscr{L}_\alpha|_{\widetilde{\mathscr{H}}^{a,b}_n}$ is equivalent to the operator
\begin{equation}
d\pi_{\frac{n}{c}}(\mathscr{L}_\alpha)=\frac{\pi^2n^2x^2}{c^2}-\frac{1}{4}\frac{d^2}{dx^2}-\frac{\pi n}{2c}\alpha.
\end{equation}
Under the change of variable
\begin{equation}
y=\sqrt{\frac{2\pi\abs{n}}{c}}x,
\end{equation}
the operator $d\pi_{\frac{n}{c}}(\mathscr{L}_\alpha)$ becomes
\begin{equation}
\frac{\pi|n|}{2c}\kak{y^2-\frac{d^2}{dy^2}}-\frac{\pi n}{2c}\alpha.
\end{equation}
As is well known, the eigenvalues of the Hamiltonian
$
y^2-\frac{d^2}{dy^2}
$
of the quantum harmonic oscillator are
$\set{2\lambda+1\mid\lambda\in\Z_{\geq 0}}$. Moreover, it is known that the function
\begin{equation}
F_\lambda(y)=H_\lambda(y)e^{-\frac{y^2}{2}} \label{Hermite_fct}
\end{equation}
is a $(2\lambda+1)$-eigenfunction, where $H_\lambda$ is the Hermite polynomial, that is, the polynomial satisfying
\begin{equation}
e^{-t^2+2yt}=\sum_{\lambda=0}^{\infty}\frac{H_\lambda(y)}{\lambda!}t^{\lambda}. \label{Hermite_poly}
\end{equation}
Thus we have the following result.
\begin{prop}[Folland \cite{folland}]\label{eigval}
Let $N$ be a lattice of $\H$ with $Z_N=\set{(0,0,cs)\mid s\in\Z}$. Suppose that $N$ is isomorphic to $N_l$ for some $l\in\Z_{>0}$.
Then the spectrum of $\mathscr{L}_\alpha$ on $N\backslash\H$ is
\begin{equation}
\Set{\frac{\pi\abs{n}}{2c}(2\lambda +1 - \alpha \sgn n) | \lambda\in\Z_{\geq 0},n \in \Z \setminus \set{0}}\cup\set{\pi^2(\mu^2+\nu^2) | (\mu,\nu)\in\Lambda'_{N}}.
\end{equation}
The multiplicity of
$
\frac{\pi\abs{n}}{2c}(2\lambda + 1 - \alpha \sgn n)
$
is $l\abs{n}$.
In particular, if $N=N_l$, the function $W^{a,b}_n(F_\lambda(\sqrt{2\pi\abs{n}}x))$ is an eigenfunction with the eigenvalue
$
\frac{\pi\abs{n}}{2}(2\lambda + 1 - \alpha \sgn n).
$
\end{prop}

\section{Heisenberg Bieberbach groups}
In this section, we introduce the concept of Heisenberg Bieberbach groups and give some examples of them.
Heisenberg Bieberbach groups are analogues of Bieberbach groups, which are discrete torsion-free subgroups $\Pi\subset\R^3\rtimes O(3)$ with compact quotient $\Pi\backslash\R^3$.
We will arrive at the definition of the Heisenberg Bieberbach group by replacing the commutative Lie group $\R^3$ by the Heisenberg group and the orthogonal group $O(3)$ by the unitary group $U(1)$.

For $A,B\in\R$ with $A^2+B^2=1$, we define the automorphism $U_{A,B}$ of $\H$ by
\begin{equation}
U_{A,B}(p,q,s)=\kak{Ap-Bq,Bp+Aq,s^2+\frac{1}{2}(AB(p^2-q^2)+(A^2-B^2-1)pq)}.
\end{equation}
The automorphism $U_{A,B}$ is the symplectic map corresponding to the unitary matrix
$
\kak{
\begin{array}{cc}
A & -B \\
B & A \\
\end{array}
}
$
and is called a \emph{unitary automorphism}.
It is clear that the group of unitary automorphisms can be identified with the unitary group $U(1)$.
Note that the Lebesgue measure on the polarized Heisenberg group is invariant under unitary automorphisms.

We consider the semidirect product $\H\rtimes U(1)$, that is, the direct product $\H\times U(1)$ with the group multiplication
$
(g,U)\cdot(h,V)=(gU(h),UV)
$
for $(g,U),(h,V)\in\H\times U(1)$. To simplify notation, we write $gU$ instead of $(g,U)$. The semidirect product $\H\rtimes U(1)$ acts on $\H$ from the left by
\begin{equation}
gU\cdot h=gU(h).
\end{equation}
We are now ready to define Heisenberg Bieberbach groups.

\begin{dfn}
A subgroup $\Gamma\subset\H\rtimes U(1)$ is a \emph{Heisenberg Bieberbach group} if $\Gamma$ is a discrete
 torsion-free subgroup and the quotient $\Gamma\backslash\H$ by the left action of $\Gamma$ is compact.
In this situation, $\Gamma\backslash\H$ is called a \emph{Heisenberg Bieberbach manifold}.
\end{dfn}

\begin{rem}
Actually, Bieberbach groups have been generalized to a broader class called \emph{almost-Bieberbach groups}, which includes Heisenberg Bieberbach groups. Let $G$ be a connected nilpotent Lie group and $C\subset\aut(G)$ be a maximal compact subgroup. A discrete torsion-free subgroup $\Gamma\subset G\rtimes C$ with compact quotient $\Gamma\backslash G$ is called an \emph{almost-Bieberbach group modeled on $G$}.
Moreover, the quotient $\Gamma\backslash G$ is called an \emph{infra-nilmanifold}.
For a deeper discussion of almost-Bieberbach groups modeled on $G$, we refer the reader to \cite{dekimpe}.
\end{rem}

One can easily see that any unitary automorphism preserves the CR structure on $\H$ and  the Folland--Stein operator $\mathscr{L}_\alpha$ is invariant under any unitary automorphism. Hence, a Heisenberg Bieberbach manifold $\Gamma\backslash\H$ has a natural CR structure and $\mathscr{L}_\alpha$ can be regarded as a differential operator on $\Gamma\backslash\H$.

We are interested in studying the eigenvalues and eigenfunctions of $\mathscr{L}_\alpha$ on $L^2(\Gamma\backslash\H)$ for a Heisenberg Bieberbach manifold. It is known that $N=\Gamma\cap\H$ is a lattice of $\H$ and $\Gamma/N$ is finite (see Theorem 2.2.1 in \cite{dekimpe}), and so the projection $N\backslash\H\map\Gamma\backslash\H$ is a finite covering. So, the irreducible decomposition of $L^2(N\backslash\H)$ we obtained in the previous section may be useful. We will see how well this idea works in the next two sections. Our approach turns out to be useful for the two examples below.

\begin{ex}\label{example}
(1) Let
\begin{equation}
\varphi=\kak{0,0,\frac{1}{2}}U_{-1,0}\in\H\rtimes U(1).
\end{equation}
The unitary automorphism $U_{-1,0}(p,q,s)=(-p,-q,s)$ is the $\pi$-rotation with respect to the coordinates $(p,q)$. For $l \in \Z_{>0}$, the subgroup
\begin{equation}
\Gamma_{2l,\pi}=\langle N_{2l},\varphi\rangle
\end{equation}
of $\H\rtimes U(1)$ is a Heisenberg Bieberbach group.
Note that $N_{2l}$ is a normal subgroup of $\Gamma_{2l,\pi}$ of index $2$, since $\varphi^2=(0,0,1)\in N_{2l}$ and $\varphi N_{2l}\varphi^{-1}=N_{2l}$.

Since it is obvious that $\Gamma_{2l,\pi}$ is discrete, we will show only the part that it is torsion-free. Suppose that $\Gamma_{2l,\pi}$ is not torsion-free. Then there exist an element $\gamma\in\Gamma_{2l,\pi}$ and $m\in\N$ such that $\gamma^m=e$, where $e=(0,0,0)$ is the identity.
Note that $\gamma \notin N_{2l}$ by the torsion-freeness of $N_{2l}$.
Now $\gamma^2\in N_{2l}$, from which we have $m=2k$ for some $k\in\N$, and so $(\gamma^2)^k=e$. Since $N_{2l}$ is torsion-free and $\gamma^2 \in N_{2l}$, we see that $\gamma^2=e$. On the other hand, $\gamma$ is not in $N_{2l}$, which implies that $\gamma=g\varphi$ for some $g=(\xi,\eta,t)\in N_{2l}$. Hence we get
\begin{equation}
\gamma^2=(0,0,2t-\xi\eta+1).
\end{equation}
Since $\eta$ is even, we have $\gamma^2\neq e$. This is a contradiction.

(2) Let
\begin{equation}
N'_{2l} = \sqrt{2l} \Z \times \sqrt{2l} \Z \times \Z,\quad\psi=\kak{0,0,\frac{1}{4}}U_{0,1}\in\H\rtimes U(1).
\end{equation}
The unitary automorphism $U_{0,1}(p,q,s)=(-q,p,s-pq)$ is the $\frac{\pi}{2}$-rotation with respect to the coordinates $(p,q)$. The subgroup
\begin{equation}
\Gamma_{2,\frac{\pi}{2}}'=\langle N'_{2l},\psi\rangle
\end{equation}
of  $\H\rtimes U(1)$ is also a Heisenberg Bieberbach group. Note that the lattice $N'_{2l}$ is isomorphic to $N_{2l}$. In fact, it is clear that  $S_{2l}(N'_{2l})=N_{2l}$, where $S_{2l}$ is the symplectic map
\begin{equation}
(p,q,s) \mapsto \kak{\frac{p}{\sqrt{2l}},\sqrt{2l}q,s}.
\end{equation}
Now we have $\psi^4=(0,0,1)\in N_{2l}$ and $\psi N'_{2l}\psi^{-1}=N'_{2l}$. Hence, $N'_{2l}$ is a normal subgroup of $\Gamma'_{2,\frac{\pi}{2}}$ of index $4$.

We will prove that $\Gamma'_{2,\frac{\pi}{2}}$ is torsion-free.
Suppose that $\Gamma'_{2l,\frac{\pi}{2}}$ is not torsion-free, and so we can take an element $\gamma\in\Gamma'_{2l,\frac{\pi}{2}}\setminus N'_{2l}$ such that $\gamma^m=e$ for some $m\in\N$.
Since $N'_{2l}$ is a normal subgroup of $\Gamma'_{2l,\frac{\pi}{2}}$ of index $4$, the element $\gamma$ is written as
$\gamma = g \psi^j$ for some $g = ( \xi,\eta,t) \in N'_{2l}$ and $j \in \set{1,2,3}$.
If $j=2$, we have $\gamma = g \psi^2 = g \varphi$. By the same argument as in (1), we see that $\gamma$ is not a torsion element.
If $j=1$, by a direct calculation, we have
\begin{align}
&\gamma^2 = g U_{0,1}(g) \psi^2, \\
&\gamma^4 = (0,0,4t + (\xi - \eta)^2 + 1). \label{gamma^4}
\end{align}
In particular, $\gamma^2 \notin N'_{2l}$ and $\gamma^4 \in N'_{2l}$.
Hence we see that $m=4k$ for some $k\in\N$, and so $(\gamma^4)^k=e$. Since $N '_{2l}$ is torsion-free, we get $\gamma^4=e$.
Now we have
\begin{equation}
4t + (\xi- \eta)^2 + 1 \equiv 1,\ 3 \bmod 4. \label{mod4}
\end{equation}
It follows from \eqref{gamma^4} and \eqref{mod4} that $\gamma^4 \neq e$, which is a contradiction.
The proof for the other case is similar.
\end{ex}

\section{Eigenvalues and eigenfunctions of $\mathscr{L}_\alpha$ on $\Gamma_{2l,\pi}\backslash\mathbb{H}$}
In this section, we study the eigenvalues and eigenfunctions of the Folland--Stein operator $\mathscr{L}_\alpha$ on $\Gamma_{2l,\pi}\backslash\H$ for $l\in\Z_{>0}$.

Recall from Example \ref{example} (1) that
\begin{equation}
\Gamma_{2l,\pi}=\langle N_{2l},\varphi\rangle,\quad\varphi=\kak{0,0,\frac{1}{2}}U_{-1,0}.
\end{equation}
Since functions on the Heisenberg Bieberbach manifold $\Gamma_{2l,\pi}\backslash\H$ can be identified with $\varphi$-invariant functions on the covering $N_{2l}\backslash\H$, Proposition \ref{eigval} implies that the eigenvalues of $\mathscr{L}_\alpha$ on $L^2(\Gamma_{2l,\pi}\backslash\H)$ must take the values in
\begin{equation}
\Set{\frac{\pi\abs{n}}{2}(2\lambda + 1 - \alpha \sgn n) | \lambda\in\Z_{\geq 0},n \in \Z \setminus \set{0}}\cup\set{\pi^2(\mu^2+\nu^2) | (\mu,\nu)\in\Lambda'_{N_{2l}}}.
\end{equation}

For the Heisenberg manifold $N_{2l}\backslash\H$, by Proposition \ref{weilbrezin}, we have the decomposition
\begin{equation}
L^2(N_{2l}\backslash\H)=\mathscr{H}_0\oplus\bigoplus_{n\neq 0}\mathscr{H}_n
\end{equation}
and
\begin{equation}
\mathscr{H}_n=\bigoplus_{a\in\Z/|n|\Z,b\in\Z/2l\Z}\mathscr{H}^{a,b}
\end{equation}
for $n\neq 0$.
By Proposition \ref{eigval}, an eigenfunction $G_{n,\lambda}\in\mathscr{H}_n$ with the eigenvalue
\begin{equation}
\frac{\pi\abs{n}}{2}(2\lambda + 1 - \alpha \sgn n),\quad\lambda\in\Z_{\geq 0},\ n \in \Z \setminus \set{0}
\end{equation}
can be uniquely written as
\begin{equation}
G_{n,\lambda}=\sum_{a\in\Z/\abs{n}\Z,b\in\Z/2l\Z}c_{n,\lambda}^{a,b} W_n^{a,b} F_{n,\lambda} \label{WBdecomp}
\end{equation}
for some $c_{n,\lambda}^{a,b}\in\C$. Here, we write $F_{n,\lambda} (x) = F_\lambda (\sqrt{2 \pi \abs{n}} x)$, where $F_\lambda$ is the Hermite function mentioned in \eqref{Hermite_fct}. Then $\varphi^*G_{n,\lambda}=G_{n,\lambda}$ if and only if $G_{n,\lambda}$ is an eigenfunction on $\Gamma_{2l,\pi}\backslash\H$.
In order to determine such $G_{n,\lambda}$, we describe the function $\varphi^*(W_n^{a,b}F_{n,\lambda})$ in concrete terms.
\begin{prop} \label{phiW^ab}
For $n\neq 0$ and $\lambda\in\Z_{\geq 0}$, we have
\begin{itemize}
\item[(i)] $\varphi^*(W^{a,0}_nF_{n,\lambda})=e^{\pi i(n+\lambda)}W^{-a,0}_nF_{n,\lambda}$.
\item[(ii)] For $b \neq 0$, $\varphi^*(W^{a,b}_nF_{n,\lambda})=e^{\pi i(n+\lambda)}W^{-a-1,-b}_nF_{n,\lambda}$ if $n>0$, and
$\varphi^*W^{a,b}_nF_{n,\lambda}=e^{\pi i(n+\lambda)}(W^{-a+1,-b}_n F_{n,\lambda})$ if $n<0$.
\end{itemize}
\end{prop}

\begin{proof}
(i) Since
\begin{equation}
\varphi(p,q,s)=\kak{-p,-q,s+\frac{1}{2}},
\end{equation}
it follows from the definition \eqref{WB} of the Weil--Brezin transform that
\begin{align*}
\varphi^*(W_n^{a,0}F_{n,\lambda})(p,q,s)&=(W^{a,0}F_{n,\lambda})\kak{-p,-q,s+\frac{1}{2}} \\
&=e^{2\pi in(s+\frac{1}{2})} \sum_{k\in\Z} F_{n,\lambda} \kak{-p+k+\frac{j}{\abs{n}}}
e^{2\pi in\kak{k+\frac{j}{\abs{n}}}(-q)}.
\end{align*}
The last expression can be rewritten as
\begin{equation}
e^{\pi in}e^{2\pi ins}\sum_{k\in\Z} F_{n,\lambda} \kak{-p+k+\frac{j}{\abs{n}}} e^{2\pi in\kak{-k-\frac{j}{\abs{n}}}q}.
\end{equation}
Since it is clear from \eqref{Hermite_poly} that $F_{n,\lambda} (-x) = e^{\pi i\lambda} F_{n,\lambda} (x)$ for any $x\in\R$, we get
\begin{align*}
\varphi^*(W_n^{a,0}F_{n,\lambda})(&p,q,s)
=e^{\pi i(n+\lambda)} e^{2\pi ins} \sum_{k\in\Z} F_{n,\lambda} \kak{p-k-\frac{j}{\abs{n}}} e^{2\pi in\kak{-k-\frac{j}{\abs{n}}}q}.
\end{align*}
Now
\begin{equation}
-\frac{j}{\abs{n}}=-1+\frac{\abs{n}-j}{\abs{n}},
\end{equation}
which gives
\begin{align*}
\varphi^*(W_n^{a,0}F_{n,\lambda})(&p,q,s)
=e^{\pi i(n+\lambda)} e^{2\pi ins} \sum_{k\in\Z} F_{n,\lambda} \kak{p-k-1+\frac{\abs{n}-j}{\abs{n}}}
e^{2\pi in\kak{-k-1+\frac{\abs{n}-j}{\abs{n}}}q}
\end{align*}
and consequently
\begin{align*}
\varphi^*(W_n^{a,0}F_{n,\lambda})(p,q,s)
&=e^{\pi i(n+\lambda)}e^{2\pi ins}\sum_{k'\in\Z}F_{n,\lambda}\kak{p+k'+\frac{\abs{n}-j}{\abs{n}}}e^{2\pi in\kak{k'+\frac{\abs{n}-j}{\abs{n}}}q} \\
&=e^{\pi i(n+\lambda)}(W_n^{-a,0}F_{n,\lambda})(p,q,s),
\end{align*}
where $k'=-k-1$.

(ii) Suppose that $n>0$. As in the proof of (i), we obtain
\begin{align*}
\varphi^*(W_n^{a,b}F_{n,\lambda})(p,q,s)&=W_n^{a,b}F_{n,\lambda}\kak{-p,-q,s+\frac{1}{2}} \\
&=e^{2\pi in(s+\frac{1}{2})} \sum_{k\in\Z} F_{n,\lambda} \kak{-p+k+\frac{j}{n}+\frac{u}{2ln}}
e^{2\pi in\kak{k+\frac{j}{n}+\frac{u}{2ln}}(-q)} \\
&=e^{\pi i(n+\lambda)}e^{2\pi ins} \sum_{k\in\Z} F_{n,\lambda} \kak{p-k-\frac{j}{n}-\frac{u}{2ln}}
 e^{2\pi in\kak{-k-\frac{j}{n}-\frac{u}{2ln}}q}.
\end{align*}
Since
\begin{equation}
-\frac{j}{n}-\frac{u}{2ln}=-1+\frac{n-j-1}{n}+\frac{2l-u}{2ln},
\end{equation}
we conclude that
\begin{align*}
\varphi^*(W_n^{a,b} &F_{n,\lambda})(p,q,s) \\
&=e^{\pi i(n+\lambda)}e^{2\pi ins} \sum_{k\in\Z} F_{n,\lambda} \kak{p-k-1+\frac{n-j-1}{n}+\frac{2l-u}{2ln}}
e^{2\pi in\kak{-k-1+\frac{n-j-1}{n}+\frac{2l-u}{2ln}}q} \\
&=e^{\pi i(n+\lambda)}e^{2\pi ins }\sum_{k'\in\Z} F_{n,\lambda} \kak{p+k'+\frac{n-j-1}{n}+\frac{2l-u}{2ln}}
e^{2\pi in\kak{k'+\frac{n-j-1}{n}+\frac{2l-u}{2ln}}q},
\end{align*}
which equals
$
e^{\pi i(n+\lambda)}(W_n^{-a-1,-b}F_{n,\lambda})(p,q,s).
$
The proof for $n<0$ is similar.
\end{proof}
By this proposition, we can get a necessary and sufficient condition for an eigenfunction on $N_{2l}\backslash\H$ to be $\varphi$-invariant. Let $G_{n,\lambda}\in\mathscr{H}_n$ be an eigenfunction with the eigenvalue
\begin{equation}
\frac{\pi\abs{n}}{2}(2\lambda +1 - \alpha \sgn n)
\end{equation}
for $\lambda\in\Z_{\geq 0}$, $n \in \Z \setminus \set{0}.$ If the function is written as \eqref{WBdecomp},
the following corollary obviously holds.
\begin{cor}\label{cor:phig}
Let $\lambda\in\Z_{\geq 0}$.
\begin{itemize}
\item[(i)] For $n>0$, $\varphi^*G_{n,\lambda}=G_{n,\lambda}$ if and only if
$e^{\pi i(n+\lambda)}c_{n,\lambda}^{a,0}=c_{n,\lambda}^{-a,0}$ and $e^{\pi i(n+\lambda)}c_{n,\lambda}^{a,b}=c_{n,\lambda}^{-a-1,-b}$ for all $a \in \Z / \abs{n} \Z$, $0\neq b\in\Z/2l\Z$.
\item[(ii)] For $n<0$, $\varphi^*G_{n,\lambda}=G_{n,\lambda}$ if and only if
$e^{\pi i(n+\lambda)}c_{n,\lambda}^{a,0}=c_{n,\lambda}^{-a,0}$ and $e^{\pi i(n+\lambda)}c_{n,\lambda}^{a,b}=c_{n,\lambda}^{-a+1,-b}$ for all $a \in \Z / \abs{n} \Z$, $0\neq b\in\Z/2l\Z$.
\end{itemize}
\end{cor}
By Corollary \ref{cor:phig}, we get the result we have been looking for from the beginning.
\begin{thm}\label{cor:dim}
Let $\lambda\in\Z_{\geq 0}$ and $n\neq 0$. Let $\mathscr{H}_{n,\lambda}^\varphi$ be the subspace of eigenfunctions $G_{n,\lambda}\in\mathscr{H}_n$ such that $\varphi^*G_{n,\lambda}=G_{n,\lambda}$. Then we have
\begin{align*}
  \dim \mathscr{H}_{n,\lambda}^\varphi=
	\begin{cases}
	l\abs{n}+1 & \text{if}\ \abs{n}+\lambda\ \text{is even}, \\
	l\abs{n}-1 & \text{if}\ \abs{n}+\lambda\ \text{is odd}.
	\end{cases}
\end{align*}
\end{thm}

\begin{proof}
Suppose that $n>0$ and that both $n$ and $\lambda$ are even. By Corollary \ref{cor:phig} (i),
$\varphi^* G_{n,\lambda} = G_{n,\lambda}$ if and only if
\begin{equation}
c_{n,\lambda}^{1,0}=c_{n,\lambda}^{-1,0},\  c_{n,\lambda}^{2,0}=c_{n,\lambda}^{-2,0},\ \dotsc,\ c_{n,\lambda}^{\frac{n}{2}-1,0}=c_{n,\lambda}^{-\frac{n}{2}+1,0},
\end{equation}
\begin{equation}
c^{0,u}_{n,\lambda} = c^{-1,-u}_{n,\lambda},\ c_{n,\lambda}^{1,u}=c_{n,\lambda}^{-2,-u},\  c_{n,\lambda}^{2,u}=c_{n,\lambda}^{-3,-u},\ \dotsc,\ c_{n,\lambda}^{n-1,u}=c_{n,\lambda}^{0,2l-u}
\end{equation}
for $1 \leq u \leq l - 1$ and
\begin{equation}
c_{n,\lambda}^{0,l}=c_{n,\lambda}^{-1,l},\ c_{n,\lambda}^{1,l}=c_{n,\lambda}^{-2,l},\ \dotsc,\ c_{n,\lambda}^{\frac{n}{2}-1,l}=c_{n,\lambda}^{-\frac{n}{2},l}.
\end{equation}
Hence the independent coefficients are $c_{n,\lambda}^{0,0}$, $c_{n,\lambda}^{1,0}$, $\dotsc$, $c_{n,\lambda}^{\frac{n}{2},0}$,
$c_{n,\lambda}^{1,u}$, $\dotsc$, $c_{n,\lambda}^{n-1,u}$ for $0 \leq u \leq l-1$, $c_{n,\lambda}^{0,l}$, $\dotsc$, $c_{n,\lambda}^{\frac{n}{2}-1,l}$, the number of which is $ln+1$. The other cases can be proved in the same way.
\end{proof}

We can apply Theorem \ref{cor:dim} to show Weyl's law for the Folland--Stein operator on $\Gamma_{2l,\pi}\backslash\H$.
For Heisenberg manifolds, Taylor \cite{taylor}, Strichartz \cite{strichartz} and Fan--Kim--Zeytuncu
\cite{fan} obtained an asymptotic formula for the eigenvalue counting function for $\FS{\alpha}$.
More precisely, Taylor and Strichartz got an asymptotic formula for $\alpha = 0$.
Fan--Kim--Zeytuncu generalized it to $-1 \leq \alpha \leq 1$ as follows, where
$N_{N\backslash\H}(t)$ is the number of positive eigenvalues of $\FS{\alpha}$ on $N\backslash\H$ (counted with multiplicity) which are less than or equal to $t > 0$.

\begin{prop}[Fan--Kim--Zeytuncu \cite{fan}]\label{weyllattice}
Let $N$ be a lattice. For $-1 \leq \alpha \leq 1$, we have
\begin{equation}
\lim_{t\to\infty}\frac{N_{N\backslash\H}(t)}{t^2}=A_\alpha\vol(N\backslash\H),
\end{equation}
where
\begin{align*}
  A_\alpha=
	\begin{cases}
	\dfrac{1}{\pi^2} \displaystyle \int_{-\infty}^{\infty} \dfrac{x}{\sinh x} e^{- \alpha x} \,dx & \text{if} \ -1 < \alpha < 1, \\
	\dfrac{1}{2\pi^2} \displaystyle\int_{-\infty}^{\infty} \kak{\dfrac{x}{\sinh x}}^2 \,dx & \text{if} \ \alpha = \pm 1.
	\end{cases}
\end{align*}
\end{prop}

Combining Proposition \ref{weyllattice} with Theorem \ref{cor:dim}, we can get Weyl's law for the Heisenberg Bieberbach manifold $\Gamma_{2l,\pi}\backslash\H$.

\begin{thm} \label{weyllawB}
Let $N_{\Gamma_{2l,\pi}\backslash\H}(t)$ be the number of positive eigenvalues of $\FS{\alpha}$ on $\Gamma_{2l,\pi}\backslash\H$ which are less than or equal to $t > 0$. For $-1 \leq \alpha \leq 1$, we have
\begin{equation}
\lim_{t\to\infty}\frac{N_{\Gamma_{2l,\pi}\backslash\H}(t)}{t^2}=A_\alpha\vol(\Gamma_{2l,\pi}\backslash\H),
\end{equation}
where $A_\alpha$ is the same constant as in the Proposition \ref{weyllattice}.
\end{thm}

\begin{proof}
 Since one can easily see that the set
 \begin{equation}
 \Set{(p,q,s) \in \H | 0 \leq p < 1, 0 \leq q < 2l, \frac{q}{2} \leq s < \frac{q + 1}{2}}
\end{equation}
is a fundamental domain for $\Gamma_{2l,\pi}$,
 we see that $\vol(\Gamma_{2l,\pi}\backslash\H)$ is half of $\vol(N_{2l}\backslash\H)$. Hence it suffices to show that
 \begin{equation}
 \left| N_{\Gamma_{2l,\pi}\backslash\H}(t) - \frac{1}{2} N_{N_{2l}\backslash\H}(t) \right| = o(t^2) \quad \text{as}\ t \rightarrow \infty.
 \end{equation}

Following Strichartz \cite{strichartz}, we write
\begin{equation}
N_M (t) = N_M^{(a)} (t) + N_M^{(b)} (t),
\end{equation}
where
$N_M^{(a)} (t)$ is the number of positive eigenvalues $\FS{\alpha}$ on $M$ which are of the  form $\frac{\pi\abs{n}}{2}(2\lambda +1 - \alpha \sgn n)$ and less than or equal to $t$, and $N_M^{(b)} (t)$ is the number of positive eigenvalues $\FS{\alpha}$ on $M$ which are of the  form $\pi^2(\mu^2+\nu^2)$ and less than or equal to $t$,
for $M = N_{2l}\backslash\H$ or $\Gamma_{2l,\pi}\backslash\H$.
By Weyl's law for the standard Laplacian on the torus $\R^2/\Lambda_{N_{2l}}$, we have $N_M^{(b)}(t) = o(t^2)$. Hence we only need to show that
 \begin{equation}
 \left| N_{\Gamma_{2l,\pi}\backslash\H}^{(a)} (t) - \frac{1}{2} N_{N_{2l}\backslash\H}^{(a)} (t) \right| = o(t^2) \quad \text{as}\ t \rightarrow \infty.
 \end{equation}

By Proposition \ref{eigval}, for the Heisenberg manifold $N_2\backslash\H$, the multiplicity of
$
\frac{\pi\abs{n}}{2}(2\lambda + 1 - \alpha \sgn n)
$
is $2l\abs{n}$, and so we get
\begin{equation}
N_{N_{2l}\backslash\H}^{(a)} (t) = \sideset{}{^{t}}{\sum}_{n,\lambda} 2l \abs{n}, \label{N^aN_2}
\end{equation}
where the sum $\sideset{}{^{t}}{\sum}_{n,\lambda}$ is taken over all $n \in \Z \setminus \set{0}$ and $\lambda \in \Z_{\geq 0}$ such that
\begin{equation}
\frac{\pi \abs{n}}{2}(2\lambda +1- \alpha \sgn n) \leq t.
\end{equation}
On the other hand, by Theorem \ref{cor:dim}, we have
\begin{equation}
N_{\Gamma_{2l,\pi}\backslash\H}^{(a)} (t) = \sum_{(n,\lambda) \in \mathcal{E} (t)} (l\abs{n} + 1) +
\sum_{(n,\lambda) \in \mathcal{O} (t)}
(l\abs{n} - 1). \label{N^aGamma}
\end{equation}
Here, we set
\begin{align*}
&\mathcal{E}(t) = \Set{(n,\lambda) \in \Z \setminus \set{0} \times \Z_{\geq 0} | \frac{\pi \abs{n}}{2}(2\lambda + 1 - \alpha \sgn n) \leq t,\ \abs{n} + \lambda\ \text{is even}}, \\
&\mathcal{O} (t) = \Set{(n,\lambda) \in \Z \setminus \set{0} \times \Z_{\geq 0} | \frac{\pi \abs{n}}{2}(2\lambda + 1 - \alpha \sgn n) \leq t,\ \abs{n} + \lambda\ \text{is odd}}.
\end{align*}
It follows from \eqref{N^aN_2} and \eqref{N^aGamma} that
\begin{equation}
\left| N_{\Gamma_{2l,\pi}\backslash\H}^{(a)} (t) - \frac{1}{2} N_{N_{2l}\backslash\H}^{(a)} (t) \right|
= \abs{ \# \mathcal{E} (t) - \# \mathcal{O} (t)}. \label{N^a-N^a}
\end{equation}
For fixed $\lambda \in \Z_{\geq 0}$, since numbers $n \in \Z$ such that $\abs{n} + \lambda$ is even and numbers $n \in \Z$ such that $\abs{n} + \lambda$ is odd appear alternately, the difference between the number of
\begin{align*}
\Set{n \in \Z \setminus \set{0} | (n,\lambda) \in \mathcal{E} (t)}
\end{align*}
and the number of
\begin{align*}
\Set{n \in \Z \setminus \set{0} | (n,\lambda) \in \mathcal{O} (t)}
\end{align*}
is at most 2 for all $t > 0$. On the other hand, if $\frac{\pi \abs{n}}{2}(2\lambda + 1 - \alpha \sgn n) \leq t$, we have $\lambda \leq \frac{t}{\pi}$.
It follows that
\begin{equation}
\abs{\# \mathcal{E} (t) - \# \mathcal{O} (t)} \leq 2 \kak{\frac{t}{\pi} +1}. \label{N^e-N^o}
\end{equation}
Combining \eqref{N^a-N^a} with \eqref{N^e-N^o}, the proof is complete.
\end{proof}

\section{Eigenvalues and eigenfunctions of $\mathscr{L}_\alpha$ on $\Gamma'_{2l,\frac{\pi}{2}} \backslash\mathbb{H}$} \label{sectionGammaprime}

Recall from Example \ref{example} (2) that
\begin{equation}
\Gamma'_{2l,\frac{\pi}{2}}=\langle N'_{2l},\psi\rangle,\quad\psi=\kak{0,0,\frac{1}{4}}U_{0,1}
\end{equation}
and $S_{2l}(N'_{2l})=N_{2l}$, where $S_{2l}$ is the symplectic map
\begin{equation}
(p,q,s) \mapsto \kak{\frac{p}{\sqrt{2l}},\sqrt{2l}q,s}.
\end{equation}

As in the previous section, we want to find $\psi$-invariant eigenfunctions of $\FS{\alpha}$ on $L^2(N'_{2l}\backslash\H)$. By Corollary \ref{coro:LNH}, we have the decomposition
\begin{equation}
L^2(N'_{2l}\backslash\H)=\bigoplus_{(\mu,\nu)\in\Lambda'_{N'_{2l}}}\C\chi_{\mu,\nu}\oplus\bigoplus_{n\neq0,a\in\Z/\abs{n}\Z,b\in\Z/2l\Z}\widetilde{\mathscr{H}}_n^{a,b},
\end{equation}
where  $\widetilde{\mathscr{H}}^{a,b}_n=S^*_{2l}(\mathscr{H}^{a,b}_n)$ and $\mathscr{H}_n=\set{f\in L^2(N_{2l}\backslash\H)\mid f(p,q,s)=e^{2\pi ins}g(p,q)}$.
Moreover, the spectrum of $\mathscr{L}_\alpha$ on $N'_{2l}\backslash\H$ is
\begin{equation}
\Set{\frac{\pi\abs{n}}{2}(2\lambda +1 - \alpha \sgn n) | \lambda\in\Z_{\geq 0},n \in \Z \setminus \set{0}}\cup\set{\pi^2(\mu^2+\nu^2) | (\mu,\nu)\in\Lambda'_{N'_{2l}}}.
\end{equation}
We remark that the symplectic map $S_{2l}$ does not preserve the Folland--Stein operator $\FS{\alpha}$.
Hence $W^{a,b}_n F_{n,\lambda}$, which is an eigenfunction of $\FS{\alpha}$ on $N_{2l} \backslash \H$, does not
pull back to that of $\FS{\alpha}$ on $N'_{2l} \backslash \H$.
Instead, the following proposition holds.

\begin{prop}\label{eigfctN'2l}
Let
\begin{equation}
F_{n,\lambda,l}(x)=H_\lambda(2\sqrt{l\pi\abs{n}}x)e^{-2l\pi\abs{n}x^2}. \label{F_n,lambda,l}
\end{equation}
The function $(W^{a,b}_n F_{n,\lambda,l}) \circ S_{2l} \in \widetilde{\mathscr{H}}^{a,b}_n$ is an eigenfunction of $\FS{\alpha}$ on $N'_{2l}\backslash \H$ with the eigenvalue $\frac{\pi\abs{n}}{2}(2\lambda+1-\alpha\sgn n)$.
\end{prop}

\begin{proof}
The eigenfunctions of $\FS{\alpha}$ on $L^2(N'_{2l}\backslash\H)$ are of the form $F \circ S_{2l}$, where $F$ is an eigenfunction of the operator
\begin{equation}
\widetilde{\FS{\alpha}}\colon F\mapsto \FS{\alpha}(F\circ S_{2l})\circ S_{2l}^{-1}
\end{equation}
on $L^2(N_{2l}\backslash\H)$.
By a direct calculation, we have
\begin{equation}
\widetilde{\FS{\alpha}}=\frac{1}{4}(-(\widetilde{P}^2+\widetilde{Q}^2)+i\alpha S),
\end{equation}
where
\begin{equation}
\widetilde{P}=\frac{1}{\sqrt{2l}}\frac{\partial}{\partial p},\quad\widetilde{Q}=\sqrt{2l}Q+pS.
\end{equation}
Hence, one can easily see that
\begin{equation}
d \pi_n ( \widetilde{\FS{\alpha}} ) = 2 l \pi^2 n^2 x^2 - \frac{1}{8l} \frac{d^2}{dx^2} - \frac{\pi n}{2} \alpha.
\end{equation}
Under the change of variable $y = 2 \sqrt{l \pi \abs{n}}x$, the operator $d \pi_n (\widetilde{\FS{\alpha}})$ becomes
\begin{equation}
\frac{\pi \abs{n}}{2} \kak{y^2 - \frac{d^2}{dy^2}} - \frac{\pi n}{2} \alpha.
\end{equation}
Since the Hermite function $H_\lambda (y) e^{- \frac{y^2}{2}}$ is an eigenfunction of this operator with the eigenvalue $\frac{\pi\abs{n}}{2}(2\lambda+1-\alpha\sgn n)$, the function $W^{a,b}_n F_{n,\lambda,l} \in \mathscr{H}^{a,b}_n$ is an eigenfunction of $\widetilde{\FS{\alpha}}$ on $L^2(N_{2l}\backslash \H)$ with the same eigenvalue. This completes the proof.
\end{proof}

We write $\widetilde{W}^{a,b}_n F_{n,\lambda,l} = (W^{a,b}_n F_{n,\lambda,l}) \circ S_{2l}$. By Proposition \ref{eigfctN'2l}, an eigenfunction
\begin{equation}
\widetilde{G}_{n,\lambda,l} \in \widetilde{\mathscr{H}}_n = \bigoplus_{a\in\Z/\abs{n}\Z,b\in\Z/2l\Z}\widetilde{\mathscr{H}}_n^{a,b}
\end{equation}
with the eigenvalue $\frac{\pi\abs{n}}{2}(2\lambda+1-\alpha\sgn n)$ can be uniquely written as
\begin{equation}
\widetilde{G}_{n,\lambda,l}=\sum_{a\in\Z/\abs{n}\Z,b\in\Z/2l\Z} c_{n,\lambda,l}^{a,b} \widetilde{W}_n^{a,b} F_{n,\lambda,l} \label{WBdecomp2l}
\end{equation}
for some $c_{n,\lambda,l}^{a,b}\in\C$. We describe the function $\psi^*(\tilde{W}_n^{a,b}F_{n,\lambda,l})$ in concrete form in order to determine $\psi$-invariant eigenfunctions $\widetilde{G}_{n,\lambda,l}$ .
\begin{prop}\label{psiWF}
Let $n \neq 0$ and $\lambda \in \Z_{\geq 0}$, and we define $F_{n,\lambda,l}$ by \eqref{F_n,lambda,l}.
Then,
\begin{equation}
\psi^*(\widetilde{W}^{j,u}_n F_{n,\lambda,l}) = \frac{1}{\sqrt{2ln}} e^{\frac{\pi}{2} i (n + \lambda)}
\sum_{\substack{0 \leq j' \leq n-1 \\ 0 \leq u' \leq 2l - 1}} e^{- 4 l n \pi i \kak{\frac{j'}{n} + \frac{u'}{2ln}}\kak{\frac{j}{n} + \frac{u}{2ln}}} \widetilde{W}^{j',u'}_n F_{n,\lambda,l}
\end{equation}
for all $0 \leq j \leq n - 1,\ 0 \leq u \leq 2l-1$ if $n > 0$, and
\begin{equation}
\psi^*(\widetilde{W}^{j,u}_n F_{n,\lambda,l}) = \frac{1}{\sqrt{2 l \abs{n}}} e^{\frac{\pi}{2} i (n + 3 \lambda)}
\sum_{\substack{0 \leq j' \leq \abs{n} - 1 \\ 0 \leq u' \leq 2 l - 1}} e^{- 4 l n \pi i \kak{\frac{j'}{\abs{n}} + \frac{u'}{2ln}}\kak{\frac{j}{\abs{n}} + \frac{u}{2ln}}} \widetilde{W}^{j',u'}_n F_{n,\lambda,l}
\end{equation}
for all $0 \leq j \leq \abs{n} - 1,\ 0 \leq u \leq 2l-1$ if $n < 0$.
\end{prop}

\begin{proof}
Since
\begin{equation}
\psi(p,q,s)=\kak{-q,p,s+\frac{1}{4} - pq},
\end{equation}
it follows from the definition \eqref{WB} of the Weil--Brezin transform that
\begin{align}
\psi^*(\widetilde{W}_n^{j,u} F_{n,\lambda,l})(&p,q,s) \\
&=(\widetilde{W}^{j,u}_n F_{n,\lambda,l})\kak{-q,p,s+\frac{1}{4} - pq} \\
& = (W^{j,u}_n F_{n,\lambda,l}) \kak{- \frac{q}{\sqrt{2l}},\sqrt{2l} p, s + \frac{1}{4} - pq} \\
&=e^{2\pi i n s} e^{\frac{\pi}{2} i n} e^{- 2 \pi i n pq} \sum_{k\in\Z} F_{n,\lambda,l} \kak{- \frac{q}{\sqrt{2l}} + k + \frac{j}{\abs{n}}
+ \frac{u}{2 l n}} e^{2 \pi i n (k + \frac{j}{\abs{n}} + \frac{u}{2 l n}) \sqrt{2l} p}. \label{psiFn}
\end{align}
Since the lattice $N'_{2l}$ is a normal subgroup of $\Gamma'_{2l,\frac{\pi}{2}}$, the fact that that
$\widetilde{W}_n^{j,u} F_{n,\lambda,l} \in L^2(N'_{2l} \backslash \H)$ implies that $\psi^*(\widetilde{W}_n^{j,u} F_{n,\lambda,l}) \in L^2(N'_{2l} \backslash \H)$.
Consequently $\psi^*(\widetilde{W}_n^{j,u} F_{n,\lambda,l})$ has a period $\sqrt{2l}$ with respect to $q$. We want to expand $\psi^*(\widetilde{W}_n^{j,u} F_{n,\lambda,l})$ to the Fourier series with respect to $q$. To do this, we  calculate the Fourier coefficient
\begin{equation}
\mathscr{F}_m(p,s) = \frac{1}{\sqrt{2l}} \int^{\sqrt{\frac{l}{2}}}_{- \sqrt{\frac{l}{2}}} \psi^*(\widetilde{W}_n^{j,u} F_{n,\lambda,l})(p,q,s)
e^{- \frac{2 \pi i m}{\sqrt{2l}}q}\ dq
\end{equation}
for $m \in \Z$. By \eqref{psiFn}, we have
\begin{align*}
&\mathscr{F}_m(p,s)  \\
& = \frac{1}{\sqrt{2l}} e^{2\pi i n s} e^{\frac{\pi}{2} i n} \sum_{k \in \Z} \int^{\sqrt{\frac{l}{2}}}_{- \sqrt{\frac{l}{2}}}
F_{n,\lambda,l} \kak{- \frac{q}{\sqrt{2l}} + k + \frac{j}{\abs{n}}
+ \frac{u}{2 l n}} e^{2 \pi i n (-q + \sqrt{2l}(k + \frac{j}{\abs{n}} + \frac{u}{2 l n})) p}
e^{- \frac{2 \pi i m}{\sqrt{2l}}q}\ dq.
\end{align*}
Under the change of variable $Q = -q + \sqrt{2l} \kak{k + \frac{j}{\abs{n}} + \frac{u}{2ln}}$, the right-hand side can be written as
\begin{align*}
&\frac{1}{\sqrt{2l}} e^{2\pi i n s} e^{\frac{\pi}{2} i n} \sum_{k \in \Z} \int^{\sqrt{\frac{l}{2}} + \sqrt{2l} \kak{k + \frac{j}{\abs{n}} + \frac{u}{2ln}}}_{- \sqrt{\frac{l}{2}} + \sqrt{2l} \kak{k + \frac{j}{\abs{n}} + \frac{u}{2ln}}}
F_{n,\lambda,l} \kak{\frac{Q}{\sqrt{2l}}}
 e^{2 \pi i \kak{\frac{m}{\sqrt{2l}} + np} Q}
e^{- 2 \pi i m \kak{\frac{j}{\abs{n}} + \frac{u}{2ln}}}\ dQ \\
& = \frac{1}{\sqrt{2l}} e^{2\pi i n s} e^{\frac{\pi}{2} i n} e^{-2 \pi i m \kak{\frac{j}{\abs{n}} + \frac{u}{2ln}}}
\int^{\infty}_{-\infty}
F_{n,\lambda,l} \kak{\frac{Q}{\sqrt{2l}}}
 e^{2 \pi i \kak{\frac{m}{\sqrt{2l}} + np} Q}\ dQ.
\end{align*}
Since $F_{n,\lambda,l}(x)=H_\lambda(2\sqrt{l\pi\abs{n}} x)e^{-2l\pi\abs{n}x^2}$, we get
\begin{align*}
\mathscr{F}_m(p,s)
=\frac{1}{\sqrt{2l}} e^{2\pi i n s} e^{\frac{\pi}{2} i n} e^{-2 \pi i m \kak{\frac{j}{\abs{n}} + \frac{u}{2ln}}}
\int^{\infty}_{-\infty}
H_\lambda (\sqrt{2 \pi \abs{n}} Q) e^{-\pi \abs{n} Q^2}
 e^{2 \pi i \kak{\frac{m}{\sqrt{2l}} + np} Q}\ dQ.
\end{align*}
If we set  $x = \sqrt{2 \pi \abs{n}} Q$ and $\xi = -\sqrt{\frac{2 \pi}{\abs{n}}} \kak{\frac{m}{\sqrt{2l}} + np}$, the last integral is
\begin{equation}
\frac{1}{\sqrt{2 \pi \abs{n}}} \int^{\infty}_{-\infty}
H_\lambda (x) e^{-\frac{x^2}{2}}
 e^{-i \xi x}\ dx.
\end{equation}
It is known that the Fourier transform $\frac{1}{\sqrt{2 \pi}} \int^{\infty}_{-\infty}
H_\lambda (x) e^{-\frac{x^2}{2}} e^{-i \xi x}\ dx$ of the Hermite function $H_\lambda(x) e^{-\frac{x^2}{2}}$ is
$e^{\frac{3}{2} \pi i \lambda} H_\lambda(\xi) e^{-\frac{\xi^2}{2}}$. Hence it follows that
\begin{align*}
\mathscr{F}_m &(p,s) \\
& = \frac{1}{\sqrt{2 l \abs{n}}} e^{2\pi i n s} e^{\frac{\pi}{2} (n + 3 \lambda) i} e^{-2 \pi i m \kak{\frac{j}{\abs{n}} + \frac{u}{2ln}}} H_\lambda(\xi) e^{-\frac{\xi^2}{2}} \\
& = \frac{1}{\sqrt{2 l \abs{n}}} e^{2\pi i n s} e^{\frac{\pi}{2} (n + 3 \lambda) i} e^{-2 \pi i m \kak{\frac{j}{\abs{n}} + \frac{u}{2ln}}} H_\lambda \kak{-\sqrt{\frac{2 \pi}{\abs{n}}} \kak{\frac{m}{\sqrt{2l}} + np}} e^{-\frac{\pi}{\abs{n}} \kak{\frac{m}{\sqrt{2l}} + np}^2}.
\end{align*}
Note that every $m \in \Z$ is uniquely written as $m = 2 l n \kak{k' + \frac{j'}{\abs{n}} + \frac{u'}{2ln}}$
for some $k' \in \Z$, $0 \leq j' \leq \abs{n} - 1$ and $0 \leq u' \leq 2l - 1$.
Since $\frac{m}{\sqrt{2l}} + np = \sqrt{2 l} n \kak{\frac{p}{\sqrt{2 l}} + k' + \frac{j'}{\abs{n}} + \frac{u'}{2ln}}$, we have
\begin{align*}
 &H_\lambda \kak{-\sqrt{\frac{2 \pi}{\abs{n}}} \kak{\frac{m}{\sqrt{2l}} + np}} e^{-\frac{\pi}{\abs{n}} \kak{\frac{m}{\sqrt{2l}} + np}^2} \\
 & = H_\lambda \kak{-\sgn n\ 2 \sqrt{l \pi \abs{n}}\kak{\frac{p}{\sqrt{2 l}} + k' + \frac{j'}{\abs{n}} + \frac{u'}{2ln}}}
 e^{-2 l \pi \abs{n} \kak{\frac{p}{\sqrt{2 l}} + k' + \frac{j'}{\abs{n}} + \frac{u'}{2ln}}^2} \\
 & =
 \begin{cases}
	e^{i \lambda \pi} F_{n,\lambda,l} \kak{\frac{p}{\sqrt{2 l}} + k' + \frac{j'}{\abs{n}} + \frac{u'}{2ln}} & \text{if}\ n > 0, \\
	 F_{n,\lambda,l} \kak{\frac{p}{\sqrt{2 l}} + k' + \frac{j'}{\abs{n}} + \frac{u'}{2ln}} & \text{if}\ n < 0.
	\end{cases}
\end{align*}
(Here we used the fact that $H_\lambda (x)$ is an even (resp. odd) polynomial if $\lambda$ is even (resp. odd).) For $n > 0$, it follows that
\begin{equation}
\psi^*(\widetilde{W}_n^{j,u} F_{n,\lambda,l})(p,q,s) = \sum_{\substack{0 \leq j' \leq n-1 \\ 0 \leq u' \leq 2l - 1}}
\sum_{k' \in \Z} \mathscr{F}_{2 l n \kak{k' + \frac{j'}{\abs{n}} + \frac{u'}{2ln}}}(p,s)
e^{\frac{2 \pi i}{\sqrt{2l}} 2 l n \kak{k' + \frac{j'}{\abs{n}} + \frac{u'}{2ln}} q}
\end{equation}
\begin{align*}
 = \frac{1}{\sqrt{2 l \abs{n}}} e^{\frac{\pi}{2} (n + \lambda) i} & \sum_{\substack{0 \leq j' \leq n-1 \\ 0 \leq u' \leq 2l - 1}}
 e^{- 4 l n \pi i \kak{\frac{j'}{\abs{n}} + \frac{u'}{2ln}} \kak{\frac{j}{\abs{n}} + \frac{u}{2ln}}} \\
& \times e^{2\pi i n s} \sum_{k' \in \Z}
 F_{n,\lambda,l} \kak{\frac{p}{\sqrt{2 l}} + k' + \frac{j'}{\abs{n}} + \frac{u'}{2ln}}
 e^{2 \pi i n \kak{k' + \frac{j'}{\abs{n}} + \frac{u'}{2ln}} \sqrt{2l} q} ,
\end{align*}
and consequently
\begin{equation}
\psi^*(\widetilde{W}_n^{j,u} F_{n,\lambda,l}) = \frac{1}{\sqrt{2 l \abs{n}}} e^{\frac{\pi}{2} (n + \lambda) i} \sum_{\substack{0 \leq j' \leq n-1 \\ 0 \leq u' \leq 2l - 1}}
 e^{- 4 l n \pi i \kak{\frac{j'}{\abs{n}} + \frac{u'}{2ln}} \kak{\frac{j}{\abs{n}} + \frac{u}{2ln}}} \widetilde{W}^{j',u'}_n
 F_{n,\lambda,l}
\end{equation}
The proof for $n < 0$ is similar.
\end{proof}

\begin{cor}\label{cor:coefficient2l}
Let $\widetilde{G}_{n,\lambda,l} \in \widetilde{\mathscr{H}}_n$ be an eigenfunction of $\FS{\alpha}$ on $N'_{2l} \backslash \H$ with eigenvalue
$\frac{\pi \abs{n}}{2} (2 \lambda + 1 - \alpha \sgn n)$ which is written as \eqref{WBdecomp2l}
for $n \in \Z$ and $\lambda \in \Z_{\geq 0}$.
\begin{itemize}
\item[(i)] For $n > 0$, $\psi^* \widetilde{G}_{n,\lambda,l} = \widetilde{G}_{n,\lambda,l}$ if and only if
\begin{equation}
c^{j',u'}_{n,\lambda,l} = \frac{1}{\sqrt{2ln}} e^{\frac{\pi}{2} i (n + \lambda)}
\sum_{\substack{0 \leq j \leq n-1 \\ 0 \leq u \leq 2l - 1}} e^{- 4 l n \pi i \kak{\frac{j'}{n} + \frac{u'}{2ln}}\kak{\frac{j}{n} + \frac{u}{2ln}}} c^{j,u}_{n,\lambda,l} \label{eq:2l}
\end{equation}
for all $0 \leq j' \leq n - 1,\ 0 \leq u' \leq 2l-1$.
\item[(ii)] For $n < 0$, $\psi^* \widetilde{G}_{n,\lambda,l} = \widetilde{G}_{n,\lambda,l}$ if and only if
\begin{equation}
c^{j',u'}_{n,\lambda,l} = \frac{1}{\sqrt{2l\abs{n}}} e^{\frac{\pi}{2} i (n + 3 \lambda)}
\sum_{\substack{0 \leq j \leq \abs{n}-1 \\ 0 \leq u \leq 2l - 1}} e^{- 4 l n \pi i \kak{\frac{j'}{\abs{n}} + \frac{u'}{2ln}}\kak{\frac{j}{\abs{n}} + \frac{u}{2ln}}} c^{j,u}_{n,\lambda,l}
\end{equation}
for all $0 \leq j' \leq \abs{n} - 1,\ 0 \leq u' \leq 2l-1$.
\end{itemize}
\end{cor}

By Corollary \ref{cor:coefficient2l}, we can compute the dimensions of eigenspaces.

\begin{thm} \label{dim2ln}
For $\lambda\in\Z_{\geq 0}$ and $n\neq 0$, let $\widetilde{\mathscr{H}}_{n,\lambda}^\psi$ be the subspace of $\frac{\pi \abs{n}}{2} (2 \lambda + 1 - \alpha \sgn n)$-eigenfunctions $\widetilde{G}_{n,\lambda,l}\in\widetilde{\mathscr{H}}_n$ such that $\psi^*\widetilde{G}_{n,\lambda,l}=\widetilde{G}_{n,\lambda,l}$.
\begin{itemize}
\item[(i)] If $\abs{n}$ is even or $l$ is even, we have
\begin{align*}
  \dim \widetilde{\mathscr{H}}_{n,\lambda}^\psi=
	\begin{dcases}
	\dfrac{l\abs{n}}{2}+1 & \text{if}\ \abs{n}+\lambda \equiv 0 \mod 4, \\
	\dfrac{l\abs{n}}{2} & \text{if}\ \abs{n}+\lambda \equiv 1,2 \mod 4, \\
	\dfrac{l\abs{n}}{2}-1 & \text{if}\ \abs{n}+\lambda \equiv 3 \mod 4.
	\end{dcases}
\end{align*}
\item[(ii)] If $\abs{n}$ is odd and $l$ is odd, we have
\begin{align*}
  \dim \widetilde{\mathscr{H}}_{n,\lambda}^\psi=
	\begin{dcases}
	\dfrac{l\abs{n}}{2}+\dfrac{1}{2} & \text{if}\ \abs{n}+\lambda \equiv 0,2 \mod 4, \\
	\dfrac{l\abs{n}}{2}-\dfrac{1}{2} & \text{if}\ \abs{n}+\lambda \equiv 1,3 \mod 4.
	\end{dcases}
\end{align*}
\end{itemize}
\end{thm}

\begin{proof}
Let
\begin{equation}
\widetilde{\mathscr{H}}_{n,\lambda} = \spn \Set{\widetilde{W}^{a,b}_n F_{n,\lambda,l}}_{a \in \Z /\abs{n} \Z,b \in \Z / 2l \Z}.
\end{equation}
We consider the representation $\rho$ of the group $\Gamma'_{2l,\frac{\pi}{2}} / N'_{2l} \simeq \langle \psi \rangle \simeq \Z / 4\Z$ on $\widetilde{\mathscr{H}}_{n,\lambda}$ defined by
\begin{equation}
\rho(\psi)(\widetilde{W}^{a,b}_n F_{n,\lambda,l}) = \psi^* \widetilde{W}^{a,b}_n F_{n,\lambda,l}.
\end{equation}
Then note that
\begin{equation}
 \dim \widetilde{\mathscr{H}}_{n,\lambda}^\psi = \frac{1}{4}\sum_{m=0}^3 \chi_\rho(\psi^m),
\end{equation}
where $\chi_\rho$ is the character of $\rho$.
Now it is easy to see that
\begin{equation}
\chi_\rho(\psi^0) = 2 l \abs{n},\  \chi_\rho(\psi^3) = \chi_\rho(\psi^{-1}) = \overline{\chi_\rho(\psi)}.
\end{equation}
Let us suppose that $n > 0$, for the other case can be treated similarly.
First, we show that
\begin{equation}
\chi_\rho(\psi^2) = \chi_\rho(\varphi) = 2 e^{\pi(n + \lambda) i}.
\end{equation}
By Proposition \ref{phiW^ab}, we have
\begin{equation}\label{eq:20}
\varphi^*(\widetilde{W}^{a,0}_n F_{n,\lambda,l})=e^{\pi i(n+\lambda)}\widetilde{W}^{-a,0}_n F_{n,\lambda,l}
\end{equation}
and
\begin{equation}\label{eq:21}
\varphi^*(\widetilde{W}^{a,b}_n F_{n,\lambda,l})=e^{\pi i(n+\lambda)}\widetilde{W}^{-a-1,-b}_n F_{n,\lambda,l},\quad b \neq 0.
\end{equation}
To compute $\chi(\varphi)$, it suffices to find $a,b$ such that $\varphi^*(\widetilde{W}^{a,b}_n F_{n,\lambda,l})=e^{\pi i(n+\lambda)}\widetilde{W}^{a,b}_n F_{n,\lambda,l}$.
If $n$ is even, then $-a \equiv a$ mod $n$ $\Longleftrightarrow$ $a \equiv 0,\frac{n}{2}$ mod $n$.
Hece it follows from \eqref{eq:20} that
\begin{equation}
\varphi^*(\widetilde{W}^{a,0}_n F_{n,\lambda,l})=e^{\pi i(n+\lambda)}\widetilde{W}^{a,0}_n F_{n,\lambda,l}
\Longleftrightarrow a \equiv 0,\frac{n}{2} \mod n.
\end{equation}
Moreover, since $-a-1 \not\equiv a$ mod $n$ for all $a$, we get
\begin{equation}
\varphi^*(\widetilde{W}^{a,b}_n F_{n,\lambda,l}) \neq e^{\pi i(n+\lambda)}\widetilde{W}^{a,b}_n F_{n,\lambda,l}
\end{equation}
for all $a$ and $b \neq 0$ by \eqref{eq:21}. Thus, we see that $\chi(\varphi) = 2e^{\pi(n+\lambda)}$.
If $n$ is odd, a similar argument shows that
\begin{multline}
\varphi^*(\widetilde{W}^{a,b}_n F_{n,\lambda,l})=e^{\pi i(n+\lambda)}\widetilde{W}^{a,b}_n F_{n,\lambda,l}
\\
\Longleftrightarrow
a \equiv \frac{n-1}{2}\mod n,\ b\equiv 0 \mod 2l \\
\text{or}\quad a \equiv \frac{n-1}{2}\mod n,\ b \equiv l \mod 2l,
\end{multline}
and hence $\chi(\varphi) = 2e^{\pi(n+\lambda)}$.

Next, we calculate $\chi_\rho(\psi)$. By Proposition \ref{psiWF}, we have
\begin{align*}
\chi_\rho(\psi) &= \frac{1}{\sqrt{2ln}} e^{\frac{\pi}{2} i (n +  \lambda)}
\sum_{\substack{0 \leq j \leq n-1 \\ 0 \leq u \leq 2l - 1}} e^{- 4 l n \pi i \kak{\frac{j}{n} + \frac{u}{2ln}}^2} \\
&= \frac{1}{\sqrt{2ln}} e^{\frac{\pi}{2} i (n +  \lambda)}
\sum_{k=0}^{2 l n -1} e^{- \frac{2 \pi i }{2 l n}k^2}.
\end{align*}
If we set $\zeta = e^{\frac{2 \pi i }{2 l n}}$, the sum
\begin{equation}
\sum_{k=0}^{2 l n -1} \zeta^{k^2}
\end{equation}
is called the Gauss sum and is known to be
\begin{align}\label{eqGauss}
  \sum_{k=0}^{2 l n -1} \zeta^{k^2}=
	\begin{cases}
	(1 + i) \sqrt{2ln} & 2ln \equiv 0 \mod 4, \\
	\sqrt{2ln} & 2ln \equiv 1 \mod 4, \\
	0 & 2ln \equiv 2 \mod 4, \\
	i \sqrt{2ln} & 2ln \equiv 3 \mod 4.
	\end{cases}
\end{align}
However, we remark that it can only happen that $2ln \equiv 0$ mod 4 and $2ln \equiv 2$ mod 4.

Suppose that $2ln \equiv 0$ mod 4, which is equivalent to that $n$ is even or $l$ is even.
If we moreover assume that $n + \lambda \equiv 0$ mod 4, we have
\begin{equation}
\chi_\rho(\psi)=\frac{1}{\sqrt{2ln}} e^{\frac{\pi}{2} i (n +  \lambda)} \sum_{k=0}^{2 l n -1} \bar{\zeta}^{k^2}
=1-i
\end{equation}
and
\begin{equation}
\chi_\rho(\psi^2) = 2 e^{\pi(n + \lambda) i} = 2.
\end{equation}
It follows that
\begin{align}
\dim \widetilde{\mathscr{H}}_{n,\lambda}^\psi &= \frac{1}{4}(\chi_\rho(\psi^0) + \chi_\rho(\psi^1) + \chi_\rho(\psi^2) + \chi_\rho(\psi^3)) \\
&=\frac{1}{4}(2ln + (1-i) + 2 + (1+i)) \\
&=\frac{ln}{2} + 1
\end{align}
The cases where $n+\lambda \equiv 1,2,3$ mod $4$ can be computed in the same way.

Assume that $2ln \equiv 2$ mod 4, which is equivalent to that  $n$ and $l$ is odd.
Since we get $\chi(\psi) = 0$ from \eqref{eqGauss},  it follows that
\begin{align}
\dim \widetilde{\mathscr{H}}_{n,\lambda}^\psi &= \frac{1}{4}\kak{2ln + 2e^{\pi i(n+\lambda)}} \\
&=
\begin{dcases}
	\frac{l n}{2}+\frac{1}{2} & \text{if}\ \abs{n}+\lambda \equiv 0,2 \mod 4, \\
	\frac{l n}{2}-\frac{1}{2} & \text{if}\ \abs{n}+\lambda \equiv 1,3 \mod 4.
	\end{dcases}
\end{align}
\end{proof}

Finally, we can also prove Weyl's law for the Heisenberg Bieberbach manifold $\Gamma'_{2l,\frac{\pi}{2}}\backslash\H$.

\begin{thm} \label{weyllawBpsi}
Let $N_{\Gamma'_{2l,\frac{\pi}{2}}\backslash\H}(t)$ be the number of positive eigenvalues of $\FS{\alpha}$ on $\Gamma'_{2l,\frac{\pi}{2}}\backslash\H$ which are less than or equal to $t > 0$. For $-1 \leq \alpha \leq 1$, we have
\begin{equation}
\lim_{t\to\infty}\frac{N_{\Gamma'_{2l,\frac{\pi}{2}}\backslash\H}(t)}{t^2}=A_\alpha\vol(\Gamma'_{2l,\frac{\pi}{2}}\backslash\H),
\end{equation}
where $A_\alpha$ is the same constant as in the Proposition \ref{weyllattice}.
\end{thm}

\begin{proof}
Let $N_M^{(a)} (t)$ be the number of positive eigenvalues $\FS{\alpha}$ on $M$ which are of the  form $\frac{\pi\abs{n}}{2}(2\lambda +1 - \alpha \sgn n)$ and less than or equal to $t$
for $M = N'_{2l}\backslash\H$ or $\Gamma'_{2l,\frac{\pi}{2}}\backslash\H$.
By the same argument as in the proof of Theorem \ref{weyllawB}, it suffices to show that
\begin{equation}\label{eq1}
\frac{N^{(a)}_{\Gamma'_{2l,\frac{\pi}{2}}\backslash\H}(t)}{N^{(a)}_{N'_{2l}\backslash\H}(t)} \rightarrow \frac{1}{4},\quad t\rightarrow \infty.
\end{equation}
Using Theorem \ref{dim2ln}, we have
\begin{equation}
	\sideset{}{^t}{\sum}_{n,\lambda}
		\kak{\frac{l \abs{n}}{2} - 1} 
	\leq N^{(a)}_{\Gamma'_{2l,\frac{\pi}{2}}\backslash\H}(t) 
	\leq \sideset{}{^{t}}{\sum}_{n,\lambda} 
	\kak{\frac{l \abs{n}}{2} + 1},
\end{equation}
where the sum $\sideset{}{^{t}}{\sum}_{n,\lambda}$ is taken over all $n \in \Z \setminus \set{0}$ and $\lambda \in \Z_{\geq 0}$ such that
\begin{equation}
\frac{\pi \abs{n}}{2}(2\lambda +1- \alpha \sgn n) \leq t.
\end{equation}
Since $N^{(a)}_{N'_{2l}\backslash\H}(t) = \sideset{}{^{t}}{\sum}_{n,\lambda} 2l\abs{n}$, it follows that
\begin{equation}
\frac{1}{4} - \frac{\sideset{}{^{t}}{\sum}_{n,\lambda} 1}{\sideset{}{^{t}}{\sum}_{n,\lambda} 2l\abs{n}}  
\leq \frac{N^{(a)}_{\Gamma'_{2l,\frac{\pi}{2}}\backslash\H}(t)}{N^{(a)}_{N'_{2l}\backslash\H}(t)} 
\leq \frac{1}{4} + \frac{\sideset{}{^{t}}{\sum}_{n,\lambda} 1}{\sideset{}{^{t}}{\sum}_{n,\lambda} 2l\abs{n}}.
\end{equation}
Therefore, it is sufficient to show that
	\begin{equation}\label{limit}
		\frac{\sideset{}{^{t}}{\sum}_{n,\lambda} 1}{\sideset{}{^{t}}{\sum}_{n,\lambda} 2l\abs{n}} 
		\rightarrow 0,\quad t\rightarrow \infty.
	\end{equation}
For this, it is enough to check that the orders of $\sideset{}{^{t}}{\sum}_{n,\lambda} 1$ and $\sideset{}{^{t}}{\sum}_{n,\lambda} 2l\abs{n}$ with respect to $t$ are $t \log t$ and $t^{2}$, respectively.
By the symmetry of $\sgn n$, we may assume that $0 \leq \alpha \leq 1$ without loss of generality. 
We consider the cases $0 \leq \alpha < 1$ and $\alpha = 1$ separately.

We first assume that $0 \leq \alpha <1$. We have
	\begin{equation}
		\sideset{}{^{t}}{\sum}_{n,\lambda} 1 
		= \sum_{\lambda=0}^{\left\lfloor \frac{2}{\pi}t \right\rfloor}\left\lfloor 
			\frac{2}{\pi} \frac{t}{2\lambda + 1 - \alpha} 
		\right\rfloor 
		+ \sum_{\lambda=0}^{\left\lfloor \frac{2}{\pi}t \right\rfloor}\left\lfloor 
			\frac{2}{\pi} \frac{t}{2\lambda + 1 + \alpha} 
		\right\rfloor
	\end{equation}
and
	\begin{align}
		& \sideset{}{^{t}}{\sum}_{n,\lambda} 2l\abs{n} \\
		 & = 2l\left(
			\sum_{\lambda = 0}^{\left\lfloor \frac{2}{\pi}t \right\rfloor}
			\sum_{n = 1}^{\left\lfloor\frac{2}{\pi}\frac{t}{2\lambda + 1 - \alpha}\right\rfloor}n
			+ \sum_{\lambda = 0}^{\left\lfloor \frac{2}{\pi}t \right\rfloor}
			\sum_{n = 1}^{\left\lfloor\frac{2}{\pi}\frac{t}{2\lambda + 1 + \alpha}\right\rfloor}n
		\right) \\
		& = 2l\left(
			\sum_{\lambda = 0}^{\left\lfloor\frac{2}{\pi}t\right\rfloor}
			\frac{1}{2}\left\lfloor \frac{2}{\pi}\frac{t}{2\lambda + 1 - \alpha} \right\rfloor \left(
					\left\lfloor \frac{2}{\pi}\frac{t}{2\lambda + 1 - \alpha} \right\rfloor + 1
				\right)
			+ \sum_{\lambda = 0}^{\left\lfloor\frac{2}{\pi}t\right\rfloor}
			\frac{1}{2}\left\lfloor \frac{2}{\pi}\frac{t}{2\lambda + 1 + \alpha} \right\rfloor \left(
					\left\lfloor \frac{2}{\pi}\frac{t}{2\lambda + 1 + \alpha} \right\rfloor + 1
				\right)
		\right) \\
		& = l\left(
			\sum_{\lambda = 0}^{\left\lfloor\frac{2}{\pi}t\right\rfloor}
			\left\lfloor \frac{2}{\pi}\frac{t}{2\lambda + 1 - \alpha} \right\rfloor \left(
					\left\lfloor \frac{2}{\pi}\frac{t}{2\lambda + 1 - \alpha} \right\rfloor + 1
				\right)
			+ \sum_{\lambda = 0}^{\left\lfloor\frac{2}{\pi}t\right\rfloor}
			\left\lfloor \frac{2}{\pi}\frac{t}{2\lambda + 1 + \alpha} \right\rfloor \left(
					\left\lfloor \frac{2}{\pi}\frac{t}{2\lambda + 1 + \alpha} \right\rfloor + 1
				\right)
		\right).
	\end{align}
Setting 
	\begin{equation}
		F_{\pm}(t) 
		= \sum_{\lambda=0}^{\left\lfloor \frac{2}{\pi}t \right\rfloor}
		\left\lfloor \frac{2}{\pi} \frac{t}{2\lambda + 1 \mp \alpha}\right\rfloor,
		\quad
		G_{\pm}(t)
		= \sum_{\lambda = 0}^{\left\lfloor\frac{2}{\pi}t\right\rfloor}
		\left\lfloor \frac{2}{\pi}\frac{t}{2\lambda + 1 \mp \alpha} \right\rfloor \left(
			\left\lfloor \frac{2}{\pi}\frac{t}{2\lambda + 1 \mp \alpha} \right\rfloor + 1
		\right),
	\end{equation}
it follows that
	\begin{equation}
		\sum_{\lambda = 0}^{\lfloor ct \rfloor}\left(
			\frac{ct}{2\lambda + 1 \mp \alpha} - 1
		\right)
		\leq F_{\pm}(t)
		\leq\sum_{\lambda = 0}^{\lfloor ct \rfloor}
		\frac{ct}{2\lambda + 1 \mp \alpha}
	\end{equation}
and
	\begin{equation}
		\sum_{\lambda = 0}^{\lfloor ct \rfloor}\left(
			\frac{ct}{2\lambda + 1 \mp \alpha} - 1
		\right)\frac{ct}{2\lambda + 1 \mp \alpha}
		\leq G_{\pm}(t)
		\leq\sum_{\lambda = 0}^{\lfloor ct \rfloor}
		\frac{ct}{2\lambda + 1 \mp \alpha}\left(
			\frac{ct}{2\lambda + 1 \mp \alpha} + 1
		\right),
	\end{equation}
where $c = \frac{2}{\pi}$. If we set $H_{\pm}(t) = \sum_{\lambda = 0}^{\lfloor ct \rfloor}\frac{ct}{2\lambda + 1 \mp \alpha}$, the two inequalities are rewritten as
	\begin{equation}\label{sumF}
		H_{\pm}(t) - \lfloor ct \rfloor - 1
		\leq F_{\pm}(t)
		\leq H_{\pm}(t)
	\end{equation}
and
	\begin{equation}\label{sumG}
		\sum_{\lambda = 0}^{\lfloor ct \rfloor}\left(
			\frac{ct}{2\lambda + 1 \mp \alpha}
		\right)^{2} - H_{\pm}(t)
		\leq G_{\pm}(t)
		\leq\sum_{\lambda = 0}^{\lfloor ct \rfloor}\left(
			\frac{ct}{2\lambda + 1 \mp \alpha}
		\right)^{2} + H_{\pm}(t).
	\end{equation}

By a direct calculation, we get
	\begin{align}
		H_{\pm}(t) 
		= \sum_{\lambda = 0}^{\lfloor ct \rfloor}\frac{ct}{2\lambda + 1 \mp \alpha}
		& \leq ct\left(
			\frac{1}{1 \mp \alpha} 
			+ \int_{1}^{ct}\frac{1}{2x - 1 \mp \alpha}\,dx
		\right) \\
		& = ct\left(
		\frac{1}{2}\log(2ct - 1 \mp \alpha) + \frac{1}{1 \mp \alpha} - \frac{1}{2}\log(1 \mp \alpha)
		\right)
	\end{align}
and
	\begin{align}
		H_{\pm}(t)
		= \sum_{\lambda = 0}^{\lfloor ct \rfloor}
			\frac{ct}{2\lambda + 1 \mp \alpha} 
		& \geq ct\int_{0}^{ct - 1}
			\frac{1}{2x + 1 \mp \alpha} 
		\,dx \\
		& = ct\left(
		\frac{1}{2}\log(2ct - 1 \mp \alpha) - \frac{1}{2}\log(1 \mp \alpha)
		\right).
	\end{align}
Hence the orders of $F_{\pm}(t)$ are $t \log t$ by \eqref{sumF}, and so the order of the sum $\sideset{}{^{t}}{\sum}_{n,\lambda}1$ is also $t \log t$. On the other hand, we also have
	\begin{align}
		\sum_{\lambda = 0}^{\lfloor ct \rfloor}\left(
			\frac{ct}{2\lambda + 1 \mp \alpha}
		\right)^{2}
		& \leq c^{2}t^{2}\left(
			\frac{1}{(1 \mp \alpha)^{2}}
			+ \int_{1}^{ct}
				\frac{1}{(2x - 1 \mp \alpha)^{2}}\,dx
		\right) \\
		& = c^{2}t^{2}\left(
			\frac{1}{(1 \mp \alpha)^{2}}
			+ \frac{1}{2(1 \mp \alpha)}
			- \frac{1}{2(2ct - 1 \mp \alpha)}
		\right)
	\end{align}
and
	\begin{align}
		\sum_{\lambda = 0}^{\lfloor ct \rfloor}\left(
			\frac{ct}{2\lambda + 1 \mp \alpha}
		\right)^{2} 
		& \geq c^{2}t^{2}
		\int_{0}^{ct - 1}\frac{1}{(2x + 1 \mp \alpha)^{2}}\,dx \\
		& = c^{2}t^{2}\left(
			\frac{1}{2(1 \mp \alpha)}
			- \frac{1}{2(2ct - 1 \mp \alpha)}
		\right).
	\end{align}
Since the orders of $H_{\pm}(t)$ are $t \log t$, it follows from \eqref{sumG} and the above estimates of $\sum_{\lambda = 0}^{\lfloor ct \rfloor}\left(\frac{ct}{2\lambda + 1 \mp \alpha}\right)^{2} $ that the order of $G_{\pm}(t)$ is $t^{2}$, and hence the order of $\sideset{}{^{t}}{\sum}_{n,\lambda} 2l\abs{n}$ is also $t^{2}$.
	
Suppose that $\alpha = 1$. We remark that if $n > 0$ and $\lambda = 0$, the corresponding eigenvalues are 0. Since the kernel of $\FS{1}$ is infinite dimensional, we omit the case. Then we have
	\begin{equation}
		\sideset{}{^{t}}{\sum}_{n,\lambda} 1
		= \sum_{\lambda = 1}^{\lfloor ct \rfloor} \left\lfloor
			\frac{ct}{2\lambda}
		\right\rfloor
		+ \sum_{\lambda = 0}^{\lfloor ct \rfloor} \left\lfloor
			\frac{ct}{2(\lambda + 1)}
		\right\rfloor
	\end{equation}
and
	\begin{align}
		\sideset{}{^{t}}{\sum}_{n,\lambda} 2l\abs{n}
		 & = 2l\left(
			\sum_{\lambda = 1}^{\lfloor ct \rfloor}
			\sum_{n = 1}^{\left\lfloor \frac{ct}{2\lambda}\right\rfloor}n
			+ \sum_{\lambda = 0}^{\lfloor ct \rfloor}
			\sum_{n = 1}^{\left\lfloor \frac{ct}{2(\lambda + 1)}\right\rfloor}n
		\right) \\
		& = l\left(
			\sum_{\lambda = 1}^{\lfloor ct\rfloor}
			\left\lfloor \frac{ct}{2\lambda} \right\rfloor \left(
					\left\lfloor \frac{ct}{2\lambda} \right\rfloor + 1
				\right)
			+ \sum_{\lambda = 0}^{\left\lfloor ct\right\rfloor}
			\left\lfloor \frac{ct}{2(\lambda + 1)} \right\rfloor \left(
					\left\lfloor \frac{ct}{2(\lambda + 1)} \right\rfloor + 1
				\right)
		\right).
	\end{align}
These sums can be estimated in the same manner as in the previous case, and so the orders of $\sideset{}{^t}{\sum}_{n,\lambda} 1$ and $\sideset{}{^{t}}{\sum}_{n,\lambda} 2l\abs{n}$ are $t \log t$ and $t^{2}$, respectively.
This completes the proof.
\end{proof}

\section*{Acknowledgement}
I would like to express my sincere gratitude to Professor Yoshihiko~Matsumoto for suggesting the problem and a lot of helpful advice. I am deeply grateful to Professor Yuya~Takeuchi for valuable advice regarding the proof of Theorem \ref{dim2ln}.

\end{document}